\documentclass[12pt,a4paper]{amsart}
\usepackage{amssymb}
\usepackage{amsthm}
\usepackage{amsmath}
\usepackage{amsxtra}
\usepackage{latexsym}
\usepackage{mathrsfs}
\usepackage[all,cmtip]{xy}
\usepackage[all]{xy}
\usepackage{enumitem}
\usepackage{xcolor}
\usepackage{comment}
\usepackage{mathabx,epsfig}
\def\acts{\ \rotatebox[origin=c]{-90}{$\circlearrowright$}\ }
\def\racts{\ \rotatebox[origin=c]{90}{$\circlearrowleft$}\ }

\newtheorem{thm}{Theorem}[section]
\newtheorem{lem}[thm]{Lemma}
\newtheorem{conj}[thm]{Conjecture}

\newtheorem{prop}[thm]{Proposition}
\newtheorem{cor}[thm]{Corollary}
\newtheorem{eg}[thm]{Example}
\theoremstyle{definition}
\newtheorem{defn}[thm]{Definition}
\newtheorem{rmk}[thm]{Remark}
\newtheorem*{ack}{Acknowledgement}
\numberwithin{equation}{section}

\def\Q{{\mathbb Q}}
\def\R{{\mathbb R}}
\def\Z{{\mathbb Z}}

\def\Hom{\mathop{\mathrm{Hom}}\nolimits}

\DeclareMathOperator{\Pic}{Pic}

\DeclareMathOperator{\Spec}{Spec}

\DeclareMathOperator{\Supp}{Supp}

\DeclareMathOperator{\ord}{ord}

  {\end{list}}

\title[]
{Global $F$-splitting of surfaces admitting an int-amplified endomorphism}
\author{Shou Yoshikawa}
\address{Graduate school of Mathematical Sciences, the University of Tokyo, Komaba, Tokyo,
153-8914, Japan}
\email{yoshikaw@ms.u-tokyo.ac.jp}

\begin{document}

\begin{abstract}
In this paper, we study the global $F$-splitting of varieties admitting an int-amplified endomoprhism.
We prove that surfaces admitting an int-amplified endomorphism are of dense globally $F$-split type and,
in particular, of Calabi--Yau type.
\end{abstract}

\maketitle

\setcounter{tocdepth}{1}

\section{Introduction} 
Let $X$ be a normal projective variety over an algebraically closed field $k$ of characteristic zero. 
We say that an endomorphism $f \colon X \to X$ over $k$ is \emph{polarized} if there exist an ample Cartier divisor $H$ on $X$ and a positive integer $q$ such that $f^*H$ is linearly equivalent to $qH$.
Broustet and Gongyo proposed the following conjecture.
\begin{conj}[{\cite[Conjecture 1.2]{brou-gon}}]
If $X$ admits a non-trivial polarized endomorphism,
then $X$ is of Calabi--Yau type, that is,
there exists an effective $\Q$-Weil divisor $\Delta$ on $X$ such that $K_X+\Delta$ is $\Q$-linearly trivial and $(X, \Delta)$ is log canonical.
\end{conj}

They proved that this conjecture holds in the surface case, using an equivarinat minimal model program (MMP, for short).
In this paper, we discuss a generalization of their result.

An endomorphism $f \colon X \to X$ over $k$ is called \emph{int-amplified}
if there exists an ample Cartier divisor $H$ on $X$ such that $f^{*}H-H$ is ample.
In particular, every non-invertible polarized ednomorphism is int-amplified.
Meng proved in \cite{meng} that if $X$ admits an int-amplified endomorphism, then every MMP for $X$ is $f^n$-equivariant for some positive integer $n$ (see $\S$ 3 for the details).
Hence, as a natural generalization of Conjecture 1.1,
we propose the following conjecture
for the case where $X$ admits an int-amplified endomorphism.
\begin{conj}
If $X$ has an int-amplified endomorphism,
then $X$ is of Calabi--Yau type.
\end{conj}

In order to attack Conjecture 1.2, we use reduction to positive characteristic and global $F$-splitting.
Global $F$-splitting is a global property of a projective variety over a perfect field of positive characteristic defined by the splitting of the absolute Frobenius morphism. 
Via reduction to positive characteristic, global $F$-splitting makes sense in characteristic zero as well: 
$X$ is said to be of \emph{dense globally $F$-split type} if its modulo $p$ reduction is globally $F$-split for infinitely many primes $p$. 
Schwede and Smith asked in \cite[Question 7.1]{SS} whether varieties of dense globally $F$-split type are of Calabi--Yau type or not. 
Thus the following conjecture, together with an affirmative answer to the question of Schwede and Smith, implies Conjecture 1.2. 

\begin{conj}\label{main conj}
If $X$ admits an int-amplified endomorphism,
then $X$ is of dense globally $F$-split type.
\end{conj}

In this paper, we prove that Conjecture 1.3 holds in the surface case.
Since Gongyo and Takagi 
gave in \cite{gon-takagi} an affirmative answer to the question of Schwede and Smith in the surface case,
as a corollary of Conjecture 1.3, we also obtain Conjecture 1.2 in the surface case.

\begin{thm}[\textup{Theorem \ref{main thm'}}]\label{main thm}
Let $X$ be a normal surface over an algebraically closed field of characteristic zero admitting an int-amplified endomorphism.
Then $X$ is of dense globally $F$-split type, in particular, $X$ is of Calabi--Yau type.
\end{thm}

We remark that Theorem 1.4 gives an alternative proof of the result of Broustet and Gongyo.
We briefly explain how to prove Theorem 1.4.
The case where $K_X$ is $\Q$-linearly trivial is easily follows from Cascini, Meng and Zhang's result in \cite{cmz}.
If $K_X$ is not $\Q$-linearly trivial, every equivariant MMP ends up with a Mori fiber space (cf. {\cite{meng-zhang}, \cite{meng-zhang2}, \cite{meng}}).
In this case, using the following theorem, 
we may assume that $\pi \colon X \to Y$ is a Mori fiber space over a variety admitting an int-amplified endomorphism.

\begin{thm}[\textup{Theorem \ref{lifting of F-splittness'}}]\label{lifting of F-splittness}
Consider the following commutative diagram
\[
\xymatrix{
Z \ar[r]^{h} \ar@{-->}[d]_{\mu} & Z \ar@{-->}[d]^{\mu} \\
Z' \ar[r]_{h'} & Z', 
}
\]
where $Z$, $Z'$ are normal projective varieties, $\mu$ is a biratinal morphism or a small birational map, $h$ and $h'$ are int-amplified endomorphisms.
Then $Z$ is of dense globally $F$-split type if and only if so is $Z'$
\end{thm}

If $X$ is a surface and the support of the ramification divisor on $X$ contains a vertical component of $\pi$,
then $X$ is of globally $F$-split type (cf.~Lemma \ref{good fibration}).
In general, the ramification divisor of $X$ may not contain any vertical component of $\pi$ even if $Y$ is a projective line (see {\cite[Section 7]{mat-yoshi}}),
and this is the main difficulty of the proof of Theorem \ref{main thm}.
We overcome this difficulty by using the following covering theorem. 

\begin{thm}[\textup{Theorem \ref{covlem'}}]\label{covlem}
Consider the following commutative diagram
\[
\xymatrix{
X \ar[r]^{f} \ar[d]_{\pi} & X \ar[d]^{\pi} \\
Y \ar[r]_{g} & Y, 
}
\]
where $X$, $Y$ are klt $\Q$-factorial normal projective varieties, $\pi$ is a Mori fiber space and
$f$ is an int-amplified endomorphism.
Assume that the ramification divisor of $X$ does not contain any vertical component of $\pi$.  
Then we have the equivariant  commutative  diagram 
\[
\xymatrix{
f \acts X \ar[d]_{\pi}& \ar[l]_{\mu_{X}} \widetilde{X} \ar[d]^{ \widetilde{\pi}} \racts \widetilde{f} \\
g \acts Y & \ar[l]^{\mu_{Y}} A \racts g_{A},
}
\]
where $A$ is an abelian variety, $\mu_{Y}$ is a finite surjective morphism,
$ \widetilde{X}$ is a normal projective variety, $\mu_{X}$ is a finite surjective \'etale in codimension one morphism, and  
$ \widetilde{\pi}$ is an algebraic fiber space.  
Moreover, $ \widetilde{X}$ is the normalization of the main component of $X {\times}_{Y}A$. 
\end{thm}

Suppose the ramification divisor of $X$ does not contain any vertical component of $\pi$, and apply Theorem 1.6 to this $X$.
Then $X$ is of dense globally $F$-split type if and only if $\widetilde{X}$ is of dense globally $F$-split type (cf. Proposition 2.9).
Since $\widetilde{X}$ is a fiber space over an abelian variety, it is not so difficult to show that $\widetilde{X}$ is of dense globally $F$-split type.
In conclusion, we have Theorem 1.4.

As another corollary of Theorem 1.6, we obtain the following theorem.

\begin{thm}[\textup{Theorem \ref{Fano type thm'}}]\label{Fano type thm}
Let $X$ be a klt normal projective surface admitting an int-amplified endomorphism.
Then $X$ is of Fano type if and only if the \'etale fundamental group of the smooth locus of $X$ is finite.
\end{thm}

The organization of this paper is as follows.
$\S$ 2 is a preliminary section.
We summarize the statements of equivariant MMPs in $\S$ 3.
In $\S$ 4, we develope the general theory of global $F$-splitting of varieties appearing in an equivarinant MMP,
and prove Theorem \ref{lifting of F-splittness} and Theorem \ref{covlem}.
In $\S$ 5, we give the proof of \ref{main thm}.
$\S$ 5 is an only section we assume that $X$ is a surface.



\vspace{15pt}
{\bf Notation and Terminology.}
\begin{itemize}
\item A variety over a filed $k$ is a geometrically integral separated scheme of finite type over $k$.
A $k$-scheme $X$ is essentially of finite type over $k$ if $X$ is a localization of some scheme of finite type over $k$. 
A $\Q$-Cartier divisor (resp. $\R$-Cartier divisor) on a variety $X$ is an element of $(\mathrm{CDiv} X) {\otimes}_{\Z}\Q$ (resp. $(\mathrm{CDiv} X) {\otimes}_{\Z}\R$),
where $\mathrm{CDiv} X$ is the group of Cartier divisors on $X$.
When $X$ is normal, these groups are embedded in $\mathrm{Div}(X) {\otimes}_{\Z}\Q$ (resp. $\mathrm{Div}(X) {\otimes}_{\Z}\R$)
where $\mathrm{Div}(X)$ is the group of Weil divisors on $X$.
An element of $\mathrm{Div}(X) {\otimes}_{\Z}\Q$ (resp.  $\mathrm{Div}(X) {\otimes}_{\Z}\R$) is called $\Q$-Weil divisor (resp. $\R$-Weil divisor).
Linear equivalence and $\Q$-linear equivalence are denoted by $\sim$ and $\sim_{\Q}$, respectively. For $\Q$-Weil divisors $D$ and $E$, $D \sim E$ means $D-E$ is a principal divisor.
\item Let $X$ be a projective variety over an algebraically closed field.
\begin{itemize}
\item $N^{1}(X)$ is the group of Cartier divisors modulo numerical equivalence
(a Cartier divisor $D$ is numerically equivalent to zero, which is denoted by $D\equiv0$, if $(D\cdot C)=0$ for all irreducible curves $C$ on $X$).
\item $N_{1}(X)$ is the group of $1$-cycles modulo numerical equivalence
(a $1$-cycle $ \alpha$ is numerically zero if $(D\cdot \alpha)=0$ for all Cartier divisors $D$).
By definition, $N^{1}(X)$ and $N_{1}(X)$ are dual to each other.
\end{itemize}
\item A morphism $f \colon X \to X$ from a projective variety $X$ to itself is called self-morphism of $X$ or endomorphism on $X$.
If it is surjective, then it is a finite morphism.
\item A morphism $f \colon X \to Y$  between varieties is called an algebraic fiber space if $f$ is proper and $f_{*} \mathcal{O}_{X}= \mathcal{O}_{Y}$.
\item A morphism $f \colon X \to Y$  between varieties is called quasi-\'etale if $f$ is \'etale at every codimension one point of $X$.
\item Let $f \colon X \to Y$ be a finite separable surjective morphism between normal varieties. The ramification divisor of $f$ is denoted by $R_{f}$.
\item The Picard number of a projective variety $X$ is denoted by $\rho(X)$.
\item The function field of a variety $X$ is denoted by $K(X)$.
\item Let $f \colon X \to X$ be an endomorphism of a variety $X$.  A subset $S\subset X$ is called totally invariant under $f$ if $f^{-1}(S)=S$ as sets.
\item Consider the commutative diagram
\[
\xymatrix{
X \ar@{-->}[r]^{\pi} \ar[d]_{f} & Y \ar[d]^{g} \\
X \ar@{-->}[r]_{\pi} & Y,
}
\]
where $f, g$ are surjective morphisms and $\pi$ is a dominant rational map.
We write this diagram as 
\[
\xymatrix{
f \acts X \ar@{-->}[r]^{\pi}& Y \racts g.
}
\]
We say a commutative diagram is equivariant if each object is equipped with an endomorphism and the morphisms are equivariant with respect to these 
endomorphisms.
\item Let $X$ be a normal projective variety. A $K_{X}$-negative extremal ray contraction $\pi \colon X  \to Y$
is called of fiber type or a Mori fiber space if $\dim Y< \dim X$.
\end{itemize}

\begin{ack}
The author wishes to express his gratitude to his supervisor Professor Shunsuke Takagi for his encouragement, valuable advice and suggestions. He is also
grateful to Professor Yoshinori Gongyo, Professor Andreas H\"{o}ring, Professor Am\"{e}al Broustet, Professor S\'{e}bastien
Boucksom, Professor Makoto Enokizono, Dr.~Kenta Sato, Dr.~Kenta Hashizume, Dr.~Yohsuke Matsuzawa, Dr.~Sho Ejiri, Dr.~Teppei Takamatsu, Professor Osamu Fujino, Professor Noboru Nakayama for their helpful comments and suggestions. 
This work was supported by the Program for Leading Graduate Schools, MEXT, Japan.
\end{ack}

\section{Preliminaries}

\subsection{Varieties of Fano type and Calabi--Yau type}
In this paper, we use the following terminology.

\begin{defn}[{cf.~\cite[Definition 2.34]{komo}, \cite[Remark 4.2]{SS}}]
Let $X$ be a normal variety over a field $k$ of arbitary characteristic and $\Delta$ be an effective $\Q$-Weil divisor on $X$ such that $K_X+\Delta$ is $\Q$-Cartier.
$\pi \colon Y \to X$ be a birational morphism from a normal variety $Y$. 
Then we can write
\[
K_Y=\pi^*(K_X+\Delta)+\sum_{E}(a(E,X,\Delta)-1)E,
\]
where $E$ runs through all prime divisors on $Y$.
We say that the pair $(X,\Delta)$ is {\em log canonical} or {\em lc}, for short (resp.,~{\em Kawamata log terminal} or {\em klt}, for short) if $a(E,X,\Delta) \geq 0$ (resp., $a(E,X,\Delta)>0$) for every prime divisor $E$ over $X$. 
If $\Delta=0$, we simply say that $X$ is log canonical (resp., klt).
\end{defn}

\begin{defn}[cf. {\cite[Lemma-Definition 2.6]{prokshok-mainII}}]
Let $X$ be a normal projective variety over a field and $\Delta$ be an effective $\Q$-Weil divisor on $X$ such that $K_X+\Delta$ is $\Q$-Cartier.
\begin{enumerate}
    \item We say that $(X,\Delta)$ is {\em log Fano} if $-(K_X+\Delta)$ is ample $\Q$-Cartier and $(X,\Delta)$ is klt.
    We say that $X$ is of {\em Fano type} if there exists an effective $\Q$-Weil divisor $\Delta$ on $X$ such that $(X,\Delta)$ is log Fano. 
    \item We say that $(X,\Delta)$ is {\em log Calabi--Yau} if $K_X+\Delta \sim_{\Q} 0$ and $(X,\Delta)$ is log canonical.
    We say that $X$ is of {\em Calabi--Yau type} if there exists an effective $\Q$-Weil divisor $\Delta$ on $X$ such that $(X,\Delta)$ is log Calabi--Yau type.
\end{enumerate}

\end{defn}

The reader is referred to \cite[Lemma-Definition 2.6]{prokshok-mainII} for more details. 

\subsection{Globally $F$-regular and $F$-split varieties}
In this subsection, we review the definition and basic properties of {\em globally $F$-regularity} and {\em globally $F$-splitting}.

A field $k$ of prime characteristic $p$ is called {\em $F$-finite} if $[k:k^p] < \infty$.

\begin{defn}[\textup{\cite[Definition 3.1]{SS}}]
Let $X$ be a normal projective variety defined over an $F$-finite field of characteristic $p>0$.
\begin{enumerate}
    \item We say that $X$ is {\em globally $F$-split} if the Frobenius map
    \[
    \mathcal{O}_X \to F_*\mathcal{O}_X
    \]
    splits as an $\mathcal{O}_X$-module homomorphism. \\
    \item We say that $X$ is {\em globally $F$-regular} if for every effective Weil divisor $D$ on $X$, there exists $e \in \Z_{>0}$ such that the composition map
    \[
    \mathcal{O}_X \to F^e_*\mathcal{O}_X \hookrightarrow F^e_*\mathcal{O}_X(D)
    \]
    of the $e$-times iterated Frobenius map $\mathcal{O}_X \to F^e_*\mathcal{O}_X$ with a natural inclusion $F^e_*\mathcal{O}_X \hookrightarrow F^e_*\mathcal{O}_X(D)$ splits as an $\mathcal{O}_X$-module homomorphism.
\end{enumerate}

\end{defn}

Now we briefly explain how to reduce things from characteristic zero to characteristic $p > 0$. 
The reader is referred to \cite[Chapter 2]{HH} and \cite[Section 3.2]{MS} for further details. 

Let $X$ be a normal variety over a field $k$ of characteristic zero and $D=\sum_i d_i D_i$ be a $\mathbb{Q}$-Weil divisor on $X$. 
Choosing a suitable finitely generated $\mathbb{Z}$-subalgebra $A$ of $k$, 
we can construct a scheme $X_A$ of finite type over $A$ and closed subschemes $D_{i, A} \subsetneq X_A$ such that 
there exists isomorphisms 
\[\xymatrix{
X \ar[r]^{\cong \hspace*{3em}} &  X_A \times_{\Spec \, A} \Spec \, k\\
D_i \ar[r]^{\cong \hspace*{3em}} \ar@{^{(}->}[u] & D_{i, A} \times_{\Spec \, A} \Spec \, k. \ar@{^{(}->}[u]\\
}\]
Note that we can enlarge $A$ by localizing at a single nonzero element and replacing $X_A$ and $D_{i,A}$ with the corresponding open subschemes. 
Thus, applying the generic freeness \cite[(2.1.4)]{HH}, we may assume that $X_A$ and $D_{i, A}$ are flat over $\Spec \, A$.
Enlarging $A$ if necessary, we may also assume that $X_A$ is normal and $D_{i, A}$ is a prime divisor on $X_A$. 
Letting $D_A:=\sum_i d_i D_{i,A}$, we refer to $(X_A, D_A)$ as a \textit{model} of $(X, D)$ over $A$.   
Given a closed point $\mu \in \Spec \, A$, we denote by $X_{\mu}$ (resp., $D_{i, \mu}$) the fiber of $X_A$ (resp., $D_{i, A}$) over $\mu$.  
Then $X_{\mu}$ is a scheme of finite type over the residue field $\kappa(\mu)$ of $\mu$, which is a finite field.  
Enlarging $A$ if necessary, we may assume that  $X_{\mu}$ is a normal variety over $\kappa(\mu)$, $D_{i, \mu}$ is a prime divisor on $X_{\mu}$ and consequently $D_{\mu}:=\sum_i d_i D_{i, \mu}$ is a $\mathbb{Q}$-divisor on $X_{\mu}$ for all closed points $\mu \in \Spec \, A$. 

Given a morphism $f:X \to Y$ of varieties over $k$ and a model $(X_A, Y_A)$ of $(X, Y)$ over $A$,  after possibly enlarging $A$, we may assume that $f$ is induced by a morphism $f_A :X_A \to Y_A$ of schemes of finite type over $A$. 
Given a closed point $\mu \in \Spec \, A$, we obtain a corresponding morphism $f_{\mu}:X_{\mu} \to Y_{\mu}$ of schemes of finite type over $\kappa(\mu)$. 
If $f$ is projective (resp. finite), after possibly enlarging $A$, we may assume that $f_{\mu}$ is projective (resp. finite) for all closed points $\mu \in \Spec \, A$. 

We denote by $X_{\bar{\mu}}$ the base change of $X_{\mu}$ to the algebraic closure $\overline{\kappa(\mu)}$ of $\kappa(\mu)$. 
Similarly for $D_{\bar{\mu}}$ and $f_{\bar{\mu}}:X_{\bar{\mu}} \to Y_{\bar{\mu}}$.
Note that $(X_{\mu}, D_{\mu})$ is globally $F$-regular (resp. globally $F$-split)  if and only if so is $(X_{\bar{\mu}}, D_{\bar{\mu}})$. 

\begin{defn}
Let the notation be as above. Suppose that $X$ is a normal projective variety over a field of characteristic zero. 
\begin{enumerate}
\item $X$ is said to be of \textit{dense globally $F$-split type} if for a model of $X$ over a finitely generated $\mathbb{Z}$-subalgebra $A$ of $k$, there exists a dense subset of closed points $W \subseteq \Spec \, A$ such that $X_{\mu}$ is globally $F$-split for all $\mu \in W$. 
\item $X$ is said to be of \textit{globally $F$-regular type} if for a model of $X$ over a finitely generated $\mathbb{Z}$-subalgebra $A$ of $k$, there exists a dense open subset of closed points $W \subseteq \Spec \, A$ such that $X_{\mu}$ is globally $F$-regular for all $\mu \in W$. 
\end{enumerate}
This definition is independent of the choice of a model. 
\end{defn}

The following two theorems are very important in this paper.

\begin{thm}[{\cite[Theorem 5.1]{SS}}]\label{Fano implies F-regular}
Let $X$ be a normal projective variety defined over a field of characteristic zero.
If $X$ is of Fano type, then $X$ is of globally $F$-regular type.
\end{thm}

\begin{thm}[{\cite[Theorem 1.1]{gon-takagi}}]\label{conjecture holds for surface}
Let $S$ be a normal projective surface over an algebraically closed field of characteristic zero. 
If $S$ is of dense globally $F$-split type $($resp., globally $F$-regular type$)$, then it is of Calabi--Yau type $($resp., Fano type$)$
\end{thm}

\subsection{canonical modules and duality}
Now we briefly explain canonical modules and duality. 
The reader is referred to {\cite{kunz}, \cite{mr85}, \cite{fed83}, \cite{MS}, \cite{hw}, \cite{ST-fin}}  for further details.
Let $f \colon Y \to X$ be a finite surjective morphism of normal integral schemes essentially of finite type over an $F$-finite field of characteristic $p > 0$. 
In this case, $X$ and $Y$ have canonical modules $\omega_X$ and $\omega_Y$, respectively (cf. \cite{har}),
satisfying 
\[
\mathscr{H}om_{X}(f_*\mathcal{O}_Y,\omega_X) \simeq \omega_Y.
\]
Since $X$ and $Y$ are normal schemes,
there exist Weil divisors $K_X$ and $K_Y$ such that $\mathcal{O}_X(K_X) \simeq \omega_X$ and $\mathcal{O}_Y(K_Y) \simeq \omega_Y$.
We call $K_X$ and $K_Y$ canonical divisors on $X$ and $Y$, respectively, and they are uniquely determined up to linear equivalence.
By the above duality, we obtain
\[
\mathcal{O}_Y(K_Y-f^*K_X) \simeq \mathscr{H}om_{X}(f_*\mathcal{O}_Y,\mathcal{O}_X).
\]
In particular, if $\psi \colon f_*\mathcal{O}_Y \to \mathcal{O}_X$ is an $\mathcal{O}_X$-module homomorphism, then there exists a non-zero rational section $\alpha \in K(Y)$ such that
\[
D_{\psi} := K_Y -f^*K_X+\mathrm{div}_Y(\alpha) \geq 0,
\]
and regarding this as a global section of $\mathcal{O}_Y(K_Y-f^*K_X)$, this is corresponding to $\psi$ by the above isomorphism.
We say that $\psi$ is corresponding to $D_{\psi}$.
Furtheremore, the above isomorphism induces 
\[
K(Y) \simeq \mathscr{H}om_X(f_*K(Y),K(X)).
\]
Note that this depends on the choice of the canonical divisors $K_X$ and $K_Y$.
In particular, if $\psi \colon f_*K(Y) \to K(X)$ is an $\mathcal{O}_X$-module homomorphism, then we can consider the corresponding divisor
\[
D_{\psi} :=K_Y -f^*K_X+\mathrm{div}(\alpha)
\]
for some $\alpha \in K(Y)$.

Now, we obtain the following basic result.
\begin{prop}\label{first property of splitting}
Let $f \colon Y \to X$ be a finite surjective morphism of normal integral schemes essentially of finite type over an $F$-finite field of characteristic $p>0$. 
We fix canonical divisors $K_X$ and $K_Y$ on $X$ and $Y$, respectively.
Then $\mathcal{O}_X \to f_*\mathcal{O}_Y$ splits as an $\mathcal{O}_X$-module homomorphism if and only if
there exists an $\mathcal{O}_X$-module homomorphism 
\[
\psi \colon f_*K(Y) \to K(X)
\]
such that $\psi(1)=1$ and $\psi$ is corresponding to an effective divisor.
\end{prop}

\begin{eg}\label{trace split}
If $f$ is a separable morphism, then the ramification divisor $R_f$ can be defined and $R_f$ is linearly equivalent to $K_Y-f^*K_X$.
Furthermore, we obtain the following isomorphism
\[
\mathcal{O}_Y(R_f) \simeq \mathscr{H}om_X(f_*\mathcal{O}_Y,\mathcal{O}_X),
\]
and the trace map $\mathrm{Tr}$ is corresponding to $R_f$.
By using this isomorphism, for any $\psi \colon f_*K(Y) \to K(X)$, there exists a non-zero rational section $\alpha \in K(Y)$ such that $\psi=\mathrm{Tr}(\alpha \cdot \_)$,
where $\mathrm{Tr}(\alpha \cdot \_)$ sends $x \in K(Y)$ to $\mathrm{Tr}(\alpha x)$.
In particular, $\psi$ is contained in $\mathscr{H}om_X(f_*\mathcal{O}_Y,\mathcal{O}_X)$ if and only if 
\[
R_f +\mathrm{div}_Y(\alpha) \geq 0.
\]
\end{eg}

The following is a basic property of the global $F$-spliting.

\begin{prop}\label{q-etale remain F-splitness}
Let $ \rho \colon Y \to X$ be a quasi-\'etale finite surjective morphism of normal varieties over a field of characteristic zero.
Then $X$ is of globally $F$-regular type or dense globally $F$-split type if and only if so is $Y$.
\end{prop}

\begin{proof}
We take models $\rho_A \colon Y_A \to X_A$ over a finitely generated $\Z$-algebra $A$ as in $\S$ 2.2 such that
$\rho_\mu$ is quasi-\'etale finite surjective and $\mathrm{deg}(\rho_\mu)$ is coprime to $\mathrm{char}(\kappa(\mu))$ for every $\mu \in \Spec A$.
We fix $\mu \in \Spec A$, and it is enough to show that
$X_\mu$ is globally $F$-regular or globally $F$-split if and only if so is $Y_\mu$.

Since $\mathrm{deg}(\rho_\mu)$ is coprime to $\mathrm{char}(\kappa(\mu))$, $\mathcal{O}_{X_{\mu}} \to (\rho_{\mu})_* \mathcal{O}_{Y_{\mu}}$ splits by the trace map.
Hence, if $Y$ is globally $F$-split or globally $F$-regular, then so is $X$.

Let $D$ be a effective Weil divisor on $Y$ and we assume that
\[
\mathcal{O}_{X_{\mu}} \to F_*\mathcal{O}_{X_{\mu}} \to F^e_*\mathcal{O}_{X_{\mu}}((\rho_{\mu})_* D)
\]
splits for some positive integer $e$.
Note that if $D=0$ , it means that $X$ is globally $F$-split.
We take a rational section $\alpha \in K(X)$ such that
\[
(1-p^e)K_{X_\mu} + \mathrm{div}_X(\alpha) -(\rho_\mu)_* D \geq 0
\]
is corresponding to a splitting of the above homomorphism.
Since $D'=(\rho_{\mu})^* (\rho_\mu)_*D \geq D$, it is enough to show that
\[
\mathcal{O}_{Y_{\mu}} \to F_*\mathcal{O}_{Y_{\mu}} \to F^e_*\mathcal{O}_{Y_{\mu}}(D')
\]
splits.
Since $\rho_\mu$ is quasi-\'etale, $\rho_\mu^*K_{X_\mu} = K_{Y_\mu}$.
In particular, we obtain
\[
(1-p^e)K_{Y_\mu} + \mathrm{div}_Y(\alpha) -D' \geq 0,
\]
and this is corresponding to a splitting of
\[
\mathcal{O}_{Y_{\mu}} \to F_*\mathcal{O}_{Y_{\mu}} \to F^e_*\mathcal{O}_{Y_{\mu}}(D').
\]

\end{proof}

\section{Equivariant MMP}

In this section, all varieties are defined over an algebraically closed field $k$ of characteristic $0$.

Meng and Zhang established minimal model program equivariant with respect to endomorphisms in \cite{meng-zhang2}, for varieties admitting an int-amplified endomorphism.
We summarize their results that we need later.

\begin{defn}\label{def:intamp}
A surjective endomorphism $f \colon X \to X$ over $k$ of normal projective variety $X$ is called \emph{int-amplified} if 
there exists an ample Cartier divisor $H$ on $X$ such that $f^{*}H-H$ is ample.
\end{defn}

We collect basic properties of int-amplified endomorphisms in the following lemma.

\begin{lem}\label{lem:intamp}\ 
\begin{enumerate}
\item Let $X$ be a normal projective variety, $f \colon X \to X$ a surjective morphism, and $n>0$ a positive integer.
Then, $f$ is int-amplified if and only if so is $f^{n}$.

\item Let $\pi \colon X \to Y$ be a surjective morphism between normal projective varieties.
Let $f \colon X \to X$, $g \colon Y \to Y$ be surjective endomorphisms such that $\pi \circ f=g \circ \pi$.
If $f$ is int-amplified, then so is $g$. 

\item  Let $\pi \colon X \dashrightarrow Y$ be a dominant rational map between normal projective varieties of same dimension.
Let $f \colon X \to X$, $g \colon Y \to Y$ be surjective endomorphisms such that $\pi \circ f=g \circ \pi$.
Then $f$ is int-amplified if and only if so is $g$.

\item If a normal $\Q$-Gorenstein (,that is, $K_{X}$ is $\Q$-Cartier) projective variety $X$ admits an int-amplified endomorphism, then the anti-canonical divisor 
$-K_{X}$ is numerically equivalent to an effective $\Q$-Cartier divisor.

\end{enumerate}
\end{lem}
\begin{proof}
See \cite[Lemma 3.3, 3,5, 3.6, Theorem 1.5]{meng}.
\end{proof}


\begin{thm}[Meng-Zhang]\label{equiv-thm}
Let $X$ be a $\Q$-factorial normal projective variety admitting an int-amplified endomorphism.
Let $ \Delta$ be an effective $\Q$-Weil divisor on $X$ such that $(X, \Delta)$ is klt.
\begin{enumerate}
\item\label{fin-st} There are only finitely many $(K_{X}+ \Delta)$-negative extremal rays of $ \overline{NE}(X)$.
Moreover, let $f \colon X \to X$ be a surjective endomorphism of $X$.
Then every $(K_{X}+ \Delta)$-negative extremal ray is fixed by the linear map $(f^{n})_{*}$ for some $n>0$.
\item\label{end-induced}  Let $f \colon X \to X$ be a surjective endomorphism of $X$.
Let $R$ be a $(K_{X}+ \Delta)$-negative extremal ray and $\pi \colon X \to Y$ its contraction.
Suppose $f_{*}(R)=R$.
Then,
\begin{enumerate}
\item \label{induced-on-contr} $f$ induces an endomorphism $g \colon Y \to Y$ such that $g\circ \pi=\pi \circ f$;
\item \label{induced-on-flip}if $\pi$ is a flipping contraction and $X^{+}$ is the flip, the induced rational self-map $h \colon X^{+} \dashrightarrow X^{+}$
is a morphism.
\end{enumerate}
\item  In particular, for any finite sequence of $(K_{X}+ \Delta)$-MMP and for any surjective endomorphism $f \colon X \to X$,
there exists a positive integer $n>0$ such that the sequence of MMP is equivariant under $f^{n}$.
\end{enumerate}
\end{thm}
\begin{proof}
(\ref{fin-st}) is a special case of  \cite[Theorem 4.6]{meng-zhang2}.
(\ref{induced-on-contr}) is true since the contraction is determined by the ray $R$.
(\ref{induced-on-flip}) follows from \cite[Lemma 6.6]{meng-zhang}.
\end{proof}

\begin{thm}[Equivariant MMP (Meng-Zhang)]\label{equiMMP}
Let $X$ be a $\Q$-factorial klt projective variety admitting an int-amplified endomorphism.
Then for any surjective endomorphism $f \colon X \to X$, there exists a positive integer $n>0$ and 
a sequence of rational maps
\begin{align*}
X=X_{0} \dashrightarrow X_{1} \dashrightarrow \cdots \dashrightarrow X_{r}
\end{align*}
such that
\begin{enumerate}
\item $X_{i} \dashrightarrow X_{i+1}$ is either a divisorial contraction, flip, or Mori fiber space of a 
$K_{X_{i}}$-negative extremal ray;
\item there exist surjective endomorphisms $g_{i} \colon X_{i} \to X_{i}$ for $i=0, \dots ,r$
such that $g_{0}=f^{n}$ and the following diagram commutes:
\[
\xymatrix{
X_{i} \ar@{-->}[r] \ar[d]_{g_{i}} & X_{i+1} \ar[d]^{g_{i+1}}\\
X_{i} \ar@{-->}[r] & X_{i+1} ;
}
\]
\item $X_{r}$ is a $Q$-abelian variety (that is, there exists a quasi-\'etale finite surjective morphism $A \to X_{r}$ from an abelian variety $A$,
note that $X_{r}$ might be a point).
In this case, there exists a quasi-\'etale finite surjective morphism $A \to X_{r}$ from an abelian variety $A$
and an surjective endomorphism $h \colon A \to A$ such that the diagram
\[
\xymatrix{
A \ar[r]^{h} \ar[d] & A \ar[d]\\
X_{r} \ar[r]_{g_{r}} & X_{r} ;
}
\]
commutes;
\end{enumerate}
\end{thm}
\begin{proof}
This is a part of \cite[Theorem 1.2]{meng-zhang2}.
\end{proof}

\begin{rmk}
Surjective endomorphisms on a Q-abelian variety always lift to a certain quasi-\'etale cover by an abelian variety.
See \cite[Lemma 8.1 and Corollary 8.2]{cmz}, for example. The proof works over any algebraically closed field.
\end{rmk}

\section{Global $F$-splitting of varieties appearing in an equivariant MMP}

In this section, we study the relationship between global $F$-splitting of varieties appearing in a minimal model program.

\subsection{Birational case and Fano case}

First, we consider the birational case.
In this case, we obtain the following result.

\begin{thm}[\textup{Theorem \ref{lifting of F-splittness}}]\label{lifting of F-splittness'}
Consider the following commutative diagram
\[
\xymatrix{
X \ar[r]^{f} \ar@{-->}[d]_{\pi} & X \ar@{-->}[d]^{\pi} \\
Y \ar[r]_{g} & Y, 
}
\]
where $X$, $Y$ are normal projective varieties defined over an algebraically closed field $k$ of characteristic $0$, $\pi$ is a birational morphism or a small birational map, $f$ and $g$ are int-amplified endomorphisms.
Then $X$ is of dense globally $F$-split type or of globally $F$-regular if and only if so is $Y$
\end{thm}

\begin{proof}
First note that the only if part holds without the existence of int-amplified endomorphisms by {\cite[Lemma 2.14]{gost}}.
Thus we assume $\pi$ is a birational morphism.
We may assume that $X$, $Y$, $f$, $g$ and $\pi$ are defined over algebraically closed field of characteristic $p>0$ and $Y$ is globally $F$-split and $\mathrm{deg}(g)$ is coprime to $p$.
It is enough to show that $X$ is globally $F$-split to prove the first assertion.
We fix canonical divisors $K_X$ and $K_Y$ on $X$ and $Y$, respectively with $\pi_*K_X=K_Y$. 
Let $\psi \colon F_*\mathcal{O}_Y \to \mathcal{O}_Y$ be an $\mathcal{O}_Y$-module homomorphism giving a splitting of the Frobenius morphism of $Y$.
Then there exists a non-zero rational function $\alpha \in K(Y)$ such that 
\[
(1-p)K_Y+\mathrm{div}_Y(\alpha) \geq 0
\]
is corresponding to $\psi$ (see $\S$ 2.3).
Note that $\psi$ induces the homomorphism $F_*K(Y) \to K(Y)$ with $\psi(1)=1$.
Since $K(Y)=K(X)$, $\psi$ is also corresponding to
\[
(1-p)K_X+\mathrm{div}_X(\alpha).
\]
Since $\pi$ is a birational morphism, there exists an effective exceptional divisor $E$ on $X$ such that
\[
(1-p)K_X+\mathrm{div}_X(\alpha) \geq -E.
\]
By the commutative diagram in the statement of Theorem \ref{lifting of F-splittness'}, every exceptional prime divisor of $\pi$ is totally invariant under $f$.
Since $f$ is an int-amplified endomorphism, we may assume that $E \leq R_f$ by replacing with $f$ some iterate by \cite[Theorem 3.3 (2)]{meng}.
The homomorphism 
\[
\mathrm{Tr} \circ f_*\psi \colon F_*f_*K(X) \to K(X)
\]
is corresponding to
\[
(1-p)K_X+\mathrm{div}_X(\alpha)+pR_f \geq -E+pR_f \geq 0.
\]
It implies that $\mathrm{Tr} \circ f_*\psi$ defines the homomorphism $F_*f_*\mathcal{O}_X \to \mathcal{O}_X$.
Since $\mathrm{Tr} \circ f_*\psi(1)= \mathrm{deg}(f)$ is unit of $\mathcal{O}_X$, $X$ is globally $F$-split.

Next, we assume that $Y$ is globally $F$-regular.
Let $D$ be an effective Weil divisor on $X$.
There exist a non-zero rational function $\alpha \in K(Y)$ and a positive integer $e$ such that
\[
(1-p^e)K_Y+\mathrm{div}_Y(\alpha)-\pi_*D \geq 0
\]
and corresponding homomorphism $\psi_1$ gives a splitting of
\[
\mathcal{O}_Y \to F^e_*\mathcal{O}_Y \to F^e_*\mathcal{O}_Y(\pi_*D).
\]
Since $\pi$ is a birational morphism, there exists a positive integer $e'$ such that
\[
(1-p^e)K_X+\mathrm{div}_X(\alpha)-D+p^{e+e'}R_f \geq 0
\]
by the same argument as above.
Since $X$ is globally $F$-split, we can find a homomorphism $\psi_2$ giving a splitting $\mathcal{O}_X \to F^{e'}_*\mathcal{O}_X$ and corresponding to
\[
(1-p^{e'})K_X + \mathrm{div}_X(\beta) \geq 0
\]
for some non-zero rational function $\beta \in K(X)$.
Hence $\mathrm{Tr} \circ f_*\psi_2 \circ f_*F^{e'}_*\psi_1$ is corresponding to
\[
(1-p^e)K_X+\mathrm{div}_X(\alpha)-D+p^e((1-p^{e'})K_X+\mathrm{div}_X(\beta))+p^{e+e'}R_f \geq 0.
\]
Thus this homomorphism defines a homomorphism $F^{e+e'}_*f_*\mathcal{O}_X \to \mathcal{O}_X$,
and we obtain that $X$ is globally $F$-regular.
\end{proof}

\begin{cor}\label{log Fano case}
Let $X$ be a normal klt $\Q$-factorial projective variety defined over an algebraically closed field $k$ of characteristic $0$ and $X$  admits an int-amplified endomorphism.
Assume that some MMP of $X$ ends up with a point.
Then $X$ is of globally $F$-regular type.
\end{cor}

\begin{proof}
We consider a MMP of $X$ 
\[
X=X_1 \dashrightarrow X_2 \dashrightarrow \cdots \dashrightarrow X_r
\]
such that $X_r \to \Spec(k)$ is a Mori fiber space.
In particular, $X_r$ is of Fano type.
By Theorem \ref{Fano implies F-regular}, $X_r$ is of globally $F$-regular type.
Since $X$ has an int-amplified endomorphism $f$,
above MMP is $f^n$-equivariant MMP for some $n \in \Z_{>0}$ by Theorem \ref{equiv-thm}.
By Theorem \ref{lifting of F-splittness'}, $X$ is also of globally $F$-regular type.
\end{proof}

Let $S$ be a normal surface defined over a field of characteristic zero admitting an int-amplified endomorphism.
If some MMP for $S$ ends up with a point, to prove the global $F$-splitting of $S$, we may assume that $-K_S$ is ample by Theorem 4.1.
If $S$ is klt, then $S$ is globally $F$-regular type by Corollary 4.2.
If $S$ is log canonical, $S$ is of Calabi-Yau type, but global $F$-splitting is not clear.

In order to prove global $F$-splitting of $S$, we discuss the local case.

\begin{defn}
\ 
\begin{enumerate}
    \item Let $X$ be an integral scheme essentially of finite type over an $F$-finite field of positive characteristic.
    We say that $X$ is {\em $F$-pure} if $F \colon \mathcal{O}_{X ,x} \to F_*\mathcal{O}_{X,x}$ splits as $\mathcal{O}_{X,x}$-module homomorphism for every point $x \in X$. 
    \item Let $X$ be a normal integral scheme essentially of finite type over a field of characteristic zero.
    We say that $X$ is of {\em dense $F$-pure type} if taking a model $X_A$ over a finitely generated $\Z$-algebra $A$ as in $\S$ 2.2, there exists a dense subset $S$ of the closed points of $\Spec A$ such that $X_s$ is $F$-pure for every $s \in S$.
\end{enumerate}
\end{defn}

\begin{lem}\label{local endo F-pure}
Let $(R,\mathfrak{m})$ be a Noetherian local normal ring essentially of finite type over an $F$-finite field of characteristic $p>0$
and $\phi \colon R \to R$ a injective finite local homomorphism such that $p$ is coprime to $\deg(\phi)$.
Assume that $\Spec R\backslash \{\mathfrak{m}\}$ is $F$-pure and there exists a non-zero effective Cartier divisor $D$ on $\Spec R$ such that $D \leq R_{\phi}$.
Then $R$ is $F$-pure.
\end{lem}

\begin{proof}
We consider the evaluation map $\Hom_R(F_*R,R) \to R$.
Note that this map is surjective if and only if $R$ is $F$-pure.
Since $\Spec(R)\backslash \{\mathfrak{m}\}$ is $F$-pure, this evaluation map is surjective at any point of the punctured spectrum.
Thus there exists a positive integer $r$ such that the image of the evaluation map contains $\mathfrak{m}^r$.
By the assumption, there exists an element $\alpha' \in \mathfrak{m}$ such that $\mathrm{div}(\alpha') \leq R_\phi$.
Let $\alpha=\phi^{r-1}(\alpha')+\cdots +\phi(\alpha')+\alpha'$.
Then we have
\[
\mathrm{div}(\alpha)=(\phi^{r-1})^*(\alpha')+\cdots +\mathrm{div}(\alpha') \leq R_{\phi^{r}}
\]
and $\alpha \in \mathfrak{m}^r$.
We replace $\phi$ by $\phi^r$.
Since $\alpha$ is contained in the image of the evaluation map, there exists a homomorphism $\psi \colon F_*R \to R$ such that $\psi(1)=\alpha$.
Next we consider the homomorphism 
\[
\mathrm{Tr}(\alpha^{-1}\cdot \underbar \ ) \colon \phi_*K(R) \to K(R)
\]
defined by $x \in \phi_*K(R)$ to $\mathrm{Tr}(\alpha^{-1} x)$.
Note that the image of $\alpha$ by this map is an unit of $R$.
Furthermore, this map is corresponding to
\[
R_\phi -\mathrm{div}(\alpha) \geq 0.
\]
Thus $\mathrm{Tr}(\alpha^{-1}\cdot \underbar \ )$ defines the homomorphism $\psi_*R \to R$ by Example \ref{trace split}.
\\
Hence $\frac{1}{\mathrm{deg}(\phi)}F_*\mathrm{Tr}(\alpha^{-1}\cdot \underbar \ )\circ \psi$ gives a splitting.

\end{proof}

Next, we prove the following global assertion by reducing to Lemma \ref{local endo F-pure}.

\begin{prop}\label{lc Fano case}
Let $X$ be a normal projective variety defined over an algebraically closed field $k$ of characteristic $0$ with the Picard rank one and $X$ admits a non-invertible endomorphism.
Assume that $-K_X$ is ample $\Q$-Cartier divisor and $X$ has at worst rational singularities.
Furthermore assume that $X$ is of dense $F$-pure type.
Then $X$ is of dense globally $F$-split type.
\end{prop}

\begin{proof}
Let $f$ be a non-invertible endomorphism of $X$. 
Note that $f$ is a polarized endomorphism because the Picard rank of $X$ is equal to one.
There exist an ample divisor $H'$ on $X$ and a positive integer $q$ such that $f^*H' \sim qH'$.
Since $X$ is of dense $F$-pure type and $\Q$-Gorenstein, $X$ is log canonical.
By Kodaira type vanishing theorem {\cite[Corollary 2.42]{fujino}},
$\mathrm{H}^1(X,\mathcal{O}_X)=0$.
Let $\pi \colon Y \to X$ be a log resolution of $X$.
Since $X$ has at worst rational singularities, we have
\[
\mathrm{H}^1(Y,\mathcal{O}_Y) =\mathrm{H}^1(X,\mathcal{O}_X)=0.
\]
It implies that $\Pic(Y)$ is finitely generated.
Let $U$ be the maximal open subset of $X$ such that $\pi|_{\pi^{-1}(U)}$ is isomorphism.
Thus, we have
\[
\mathrm{CH}^1(X) \simeq \mathrm{CH}^1(U) \twoheadleftarrow \mathrm{CH}^1(Y) \simeq \Pic(Y),
\]
in particular, $\mathrm{CH}^1(X)$ is finitely generated.
Since $R_f$ is $\Q$-Cartier, $\{R_{f^n} \ | \ n \in \Z_{>0} \}$ is a finite set in $\mathrm{Div}(X)/\mathrm{CDiv}(X)$.
It implies that there exist positive integers $m > n >0$ such that
\[
R_{f^m} - R_{f^n}=(f^{m-1})^*R_f + \cdots (f^n)^*R_f
\]
is an effective ample Cartier divisor.
Replacing $f$ be some iterate, we may assume that there exists an effective Cartier divisor $A$ on $X$ such that $R_f \geq A$.
We define $A_n = (f^{n-1})^*A+ \cdots f^*A +A$.
Since $\mathrm{CH}^1(X)$ is finitely generated and $f^*A_n - qA_n$ is $\Q$-linearly trivial,
we have $\{f^*A_n - qA_n \ | \ n \in \Z_{>0} \}$ is a finite set.
Thus, there exist positive integers $m > n$ such that
\[
f^*A_m -qA_m -(f^*A_n-qA_n) = f^*A'-qA'
\]
is a principal divisor, where $A' = (f^{m-1})^*A + \cdots +(f^{n})^*A$.
Since $A \leq R_f$, we have
\[
A'=(f^{m-1})^*A + \cdots +(f^{n})^*A \leq (f^{m-1})^*R_f + \cdots + f^*R_f + R_f = R_{f^m}
\]
and $f^*A' \sim qA'$.
Hence replacing $f$ by some iterate, we may assume that there exists an effective ample Cartier divisor $H$ on $X$ such that
$R_f \geq H$ and $f^*H \sim qH$.
Therefore $f$ induces the graded endomorphism $\phi$ of the section ring $R=\bigoplus_{m\geq 0}\mathrm{H}^0(X,\mathcal{O}_X(mH))$.
Let $D$ be a corresponding effective Cartier divisor of $H$ on $R$.
Then $D \leq R_\phi$, since $H \leq R_f$.
Next we take a model $(X_A,H_A,f_A)$ of $(X,H,f)$ over a suitable finitely generated $\Z$-subalgebra $A$ of $k$ as in $\S$ 2.2.
Localizing $A$ at a single element, we may assume that $(R_f)_{\mu}=R_{f_\mu}$ for every $\mu \in \Spec A$.
In particular, we may assume that $H_{\mu}$ is an effective Cartier divisor such that
$H_{\mu} \leq R_{f_\mu}$ and $f^*H_{\mu} \sim qH_{\mu}$.
It means that $f_{\mu}$ induces an endomorphism $\phi_\mu$ of the section ring 
\[
R_\mu=\bigoplus_{m \geq 0} \mathrm{H}^0(X_{\mu},mH_\mu)
\]
and there exists a non-zero effective Cartier divisor $D$ on $R_\mu$ such that $D \leq R_\phi$.
Since $X$ is of dense $F$-pure type, there exists a dense subset of closed points $W \subset \Spec A$ such that $X_\mu$ is $F$-pure and $\mathrm{deg}(\phi_\mu)$ is coprime to $p$ for all $\mu \in W$.
Then $R_\mu$ satisfies the assumption of Lemma \ref{local endo F-pure},
thus $R_\mu$ is $F$-pure for all $\mu \in W$.
By {\cite[Proposition 5.3]{SS}}, $X_\mu$ is globally $F$-split for all $\mu \in W$.

\end{proof}

\subsection{Mori fiber spaces over positive dimensional varieties}
In this subsection, we consider Mori fiber spaces over a positive dimensional varieties
and we prove Theorem \ref{covlem} that plays an essential role to prove the main theorem.
In this subsection, every varieties defined over an algebraically closed field $k$ of characteristic zero.

First, we recall basic facts on main components of fiber products.

\begin{lem}\label{main comp}
Consider the following commutative diagram:
\[
\xymatrix{
\widetilde{X} \ar[r]^{ \alpha} \ar@/^18pt/[rr]^{ \widetilde{\pi}} \ar[rd]_{\mu_{X}} & X\times_{Z}A  \ar[r]^{q} \ar[d]_{p} & A \ar[d]^{\mu_{Z}}\\
& X \ar[r]_{\pi} & Z,
}
\]
where $X, Z, A, \widetilde{X}$ are normal projective varieties, $\mu_{X}, \mu_{Z}$ are finite surjective morphisms,
$ \widetilde{\pi}$ is a surjective morphism,  and $\pi$ is an algebraic fiber space.
Let $U \subset Z$ be a dense open subset such that $\mu_{Z}^{-1}(U) \to U$ is \'etale.
Then the open subscheme $p^{-1}(\pi^{-1}(U)) \subset X\times_{Z}A$ is an irreducible normal variety.
The closure $M$ of this subset in $X\times_{Z}A$ is the unique irreducible component of $X\times_{Z}A$ that dominates $X$.
We call this $M$ the main component of $X\times_{Z}A$ and equip it with the reduced structure.

Moreover,
$ \alpha$ is the normalization of $M$ if and only if the canonical homomorphism $k(X) {\otimes}_{ k(Z)} k(A) \to k( \widetilde{ X})$ is 
an isomorphism.

If $X, Z, A$ are equipped with surjective endomorphisms equivariantly, then they induce a surjective endomorphism on $ \widetilde{X}$.
\end{lem}

\begin{rmk}\label{maincompfunc}
In the setting of Lemma \ref{main comp}, the normalization of the main component is equal to the normalization of
$X$ in $k(X) {\otimes}_{k(Z)}k(A)$.
Note that since $k(X)\supset k(Z)$ is algebraically closed and $k(A) \supset k(Z)$ is a finite separable extension,
$k(X) {\otimes}_{k(Z)}k(A)$ is a field.
\end{rmk}

\begin{proof}[Proof of Lemma \ref{main comp}]
Since $\mu_{Z}^{-1}(U) \to U$ is flat, $p^{-1}(\pi^{-1}(U)) \to \pi^{-1}(U)$ is an open map 
and $p^{-1}(\pi^{-1}(U)) \to \mu_{Z}^{-1}(U)$ is an algebraic fiber space.
This implies $p^{-1}(\pi^{-1}(U))$ is irreducible.
Since $p^{-1}(\pi^{-1}(U)) \to \pi^{-1}(U)$ is \'etale and $X$ is normal, 
$p^{-1}(\pi^{-1}(U))$ is a normal variety.
This also proves that the uniqueness of the irreducible component that dominates $X$.

Since $k(X) {\otimes}_{k(Z)}k(A)$ is a field, the function field of $M$ is $k(X) {\otimes}_{k(Z)}k(A)$.
Therefore, $ \alpha$ is the normalization of $M$ if and only if $k(X) {\otimes}_{ k(Z)} k(A) \to k( \widetilde{ X})$ is an 
isomorphism.

The last statement follows from the universality of the fiber products and the uniqueness of the main component.
\end{proof}

First, we prove that if $Y$ has a finite $g$-equivariant cover by an abelian variety, then $\mu_X$ as in Theorem \ref{covlem} is quasi-\'etale.

\begin{lem}\label{quasi-etale lemma}
Let $X$ be a normal projective variety and $f \colon X \to X$ an int-amplified endomorphism and assume that there exists the following commutative diagram:
\[
\xymatrix{
\widetilde{f} \acts \widetilde{X} \ar[d]_{\mu_{X}} \ar[r]^{ \widetilde{\pi}} & \widetilde{Y} \ar[d]^{\mu_{Y}} \racts \widetilde{g} \\
f \acts X \ar[r]_{\pi} & Y \racts g,
}
\]
where
\begin{enumerate}
\item $Y$ and $\widetilde{Y}$ are normal projective varieties; 
\item $\pi$ is an algebraic fiber space, $\mu_{X}$ and $\mu_{Y}$ are finite surjective morphism;
\item $ \widetilde{X}$ is the normalization of the main component of $X \times_{Y} \widetilde{Y}$;
\item $g, \widetilde{g}, \widetilde{f}$ are int-amplified endomorphisms.
\end{enumerate}
Furthermore we assume that the following conditions:
\begin{enumerate}
\renewcommand{\labelenumi}{(\alph{enumi})}
\item\label{tilgqet} $\widetilde{g}$ is quasi-\'etale;
\item\label{imEcodim} every prime divisor $E$ on $X$ satisfies $\mathrm{codim} (\pi(E)) \leq 1$;   
\item\label{ramhor} every irreducible component $E$ of $\Supp{R_f}$ satisfies $\pi(E) = Y$
\end{enumerate}
Then $\mu_X$ is quasi-\'etale.

\end{lem}

\begin{proof}
There exists a nonempty open subset $U$ of $Y$ such that $\mu_Y^{-1}(U) \to U$ is \'etale.
Then since $\widetilde{\pi}^{-1} (\mu_Y^{-1}(U)) = \mu_X^{-1}(\pi^{-1}(U)) = \pi^{-1}(U) \times_{U} \mu_Y^{-1}(U)$ by Lemma \ref{main comp}, $\mu_X^{-1}(\pi^{-1}(U)) \to \pi^{-1}(U)$ is \'etale.
In particular, any irreducible component $E$ of $\Supp(R_{\mu_X})$ satisfies $\widetilde{\pi}(E) \neq \widetilde{Y}$.
On the other hand, we have
\[
\widetilde{f}^*(R_{\mu_X}) \leq \widetilde{f}^*(R_{\mu_X}) + R_{\widetilde{f}} = \mu_X^*(R_f) + R_{\mu_X}
\]
Since every component of $\mu_X^*(R_f)$ is horizontal by the condition $(c)$ and any component of $R_{\mu_X}$ is vertical, we obtain
\[
\widetilde{f}^*(R_{\mu_X}) \leq R_{\mu_X}
\]
It means that $\Supp{R_{\mu_X}}$ is a totally invariant divisor under $\widetilde{f}$, therefore by Lemma \ref{totinvim}, $\widetilde{\pi}(\Supp{R_{\mu_X}})$ is a totally invariant closed subset under $\widetilde{g}$.
Suppose that $\mu_X$ is not quasi-\'etale.
Let $E$ be an irreducible componet of $\Supp{R_{\mu_X}}$, then $E$ is a prime divisor.
Replacing $\widetilde{g}$ by some iterate, we may assume that $\widetilde{\pi}(E)$ is totally invariant.
By the condition $(b)$, $\mathrm{codim}(\widetilde{\pi}(E))$ is less than or equal to $1$.
Since $E$ is a vertical divisor, $\widetilde{\pi}(E)$ must be a prime divisor on $\widetilde{Y}$.
Since $\widetilde{g}$ is an int-amplified, the coefficients of $\widetilde{g}^*(\widetilde{\pi}(E))$ is grater than $2$ by \cite[Lemma 3.11]{meng}.
It means that $R_{\widetilde{g}} \geq \widetilde{\pi}(E)$,
but it is contradiction to the condition $(a)$.
Hence $\mu_X$ is quasi-\'etale. 
\end{proof}

\begin{lem}\label{totinvim}
Consider the following commutative diagram:
\[
\xymatrix{
X \ar[r]^{f} \ar[d]_{\pi} & X \ar[d]^{\pi} \\
Z \ar[r]_{h}  & Z, 
}
\]
where $X, Z$ are normal projective varieties, $\pi$ is an algebraic fiber space, and
$f, h$ are surjective endomorphisms.
Let $W\subset X$ be a closed subset.
If $f^{-1}(W) = W$ as sets, then $h^{-1}(\pi(W))=\pi(W)$.
\end{lem}
\begin{proof}
Since $h(\pi(W))=\pi(f(W))\subset \pi(W)$, we have $\pi(W)\subset h^{-1}(\pi(W))$.
Take a closed point $z\in h^{-1}(\pi(W))$.
Since $X$ is a normal variety and $f$ is a finite morphism, $f$ is an open map by going down.
Therefore, the induced map $\pi^{-1}(h^{-1}(h(z))) \to \pi^{-1}(h(z))$ is also an open map.
Thus $f(\pi^{-1}(z))$ is an open and closed subset of $\pi^{-1}(h(z))$ and therefore $f(\pi^{-1}(z))=\pi^{-1}(h(z))$.
Since $\pi^{-1}(h(z))\cap W \neq \emptyset$, we get $\emptyset \neq \pi^{-1}(z)\cap f^{-1}(W) =  \pi^{-1}(z)\cap W$, which means
$z\in \pi(W)$.
\end{proof}

\begin{lem}\label{irreducible fiber lemma}
Let $\pi \colon X \to Y$ be a Mori fiber space of normal $\Q$-factorial klt projective varieties.
Then for any prime divisor $E$ on $Y$, $\pi^{-1}(E)$ is an irreducible codimension one closed subset of $X$.
In particular, any prime divisor $F$ on $X$ satisfies $\mathrm{codim}(\pi(F)) \leq 1 $.
\end{lem}

\begin{proof}
Suppose that $\pi^{-1}(E)$ has two components.
Then we can take irreducible components $F$ and $F'$ such that $\mathrm{codim}(F) = 1$, $\pi(F)=E$ and $F \cap F' \neq \emptyset$. 
We take a closed point $y \in \pi(F' \backslash F)$, then there exist closed points $x \in F$ and $x' \in F' \backslash F $ such that $\pi(x)=\pi(x')=y$.
Since $\pi$ has connected fibers, we can take a connected curve containing $x$ and $x'$, so we can take an integral curve $C$ such that $C \cap F \neq \emptyset $ and C is not contained in $F$. 
In particular, the intersection number of $C$ and $F$ is positive.
However, taking a closed point $y' \in Y \backslash E$ and an integral curve $C' \subset \pi^{-1}(y')$ as sets, we have $(C'\cdot F) = 0 $.
It contradicts to the fact that the relative Picard rank is equal to one.

Next, we assume that there exists a prime divisor $F$ on $X$ such that $\mathrm{codim}(\pi(F)) \geq 2$.
Then we can take a prime divisor $E$ on $Y$ containing $\pi(F)$.
Hence we have $F \subset \pi^{-1}(E)$ as sets, but by the first assertion, $F = \pi^{-1}(E)$ as sets.
In particular, $\pi(F) = E$ and it is contradiction.

\end{proof}

\begin{defn}
Let $X$ be a normal variety and $\Delta$ an effective $\Q$-Weil divisor on $X$.
We say that $\Delta$ has \emph{standard coefficients} if for any prime divisor $E$ on $X$, there exists a positive integer $m$ such that $\ord_{E}(\Delta) = \frac{m-1}{m}$.

\end{defn}

\begin{lem}\label{index-one covering lem}
Let $X$ be a normal projective variety and $\Delta$ an effective $\Q$-Weil divisor on $X$ such that $K_X + \Delta $ is $\Q$-linearly equivalent to $0$.
Then there exists a finite surjective morphism $\mu \colon \widetilde{X} \to X$ from normal projective variety such that the following conditions hold:
\begin{itemize}
    \item $\mu^*(K_X +\Delta)$ is a principal divisor, that is,
    $\mu^*(K_X+\Delta) \sim 0$;
    \item if  $\mu' \colon X' \to X$ is a finite surjective morphism from a normal projective variety such that $\mu'^*(K_X+\Delta)$ is a principal divisor, then $\mu'$ factors through $\mu$.
\end{itemize}
Furthermore if $\Delta$ has standard coefficients, then $R_{\mu} = \mu^*(\Delta)$.
In particular $K_{\widetilde{X}}$ is a principal divisor.

\end{lem}

\begin{proof}
Let $m_0 = \min\{m\  |\  m(K_X+\Delta)\sim 0\}$ and take a non-zero rational section $\alpha \in K(X) =K$ such that $\mathrm{div}(\alpha)=m_0(K_X+\Delta)$.
Let $L=K[T]/(T^{m_0}-\alpha)$.
Note that $L$ is a field.
Let $\mu \colon \widetilde{X} \to X$ be the normalization of $X$ in $L$.
Then we have
\[
\mathrm{div}(T)=\frac{1}{m_0}\mathrm{div}(\alpha) = \mu^*(K_X+\Delta),
\]
so $\mu^*(K_X+\Delta)$ is a principal divisor.
Moreover, let $\mu' \colon X' \to X$ be a finite surjective morphism from a normal projective variety such that $\mu'^*(K_X+\Delta)$ is a principal divisor.
Then there exists a non-zero rational function $\beta \in K(X')=L'$ such that $\mathrm{div}(\beta)=\mu'^*(K_X+\Delta)$. 
In particular, we have $m_0 \mathrm{div}(\beta) = \mathrm{div}(\alpha)$.
Since the base field is algebraically closed, we may assume that $\beta^{m_0}=\alpha$.
It means that there exists an injective $K$-algebra homomorphism from $L$ to $L'$,
so $\mu'$ factors through $\mu$.

Next we assume that $\Delta$ has standard coefficients.
Let $E$ be a prime divisor on $X$, $m$ a positive integer, $a$ an integer such that $\ord_E(K_X+\Delta)=\frac{m-1}{m}+a$.
Let $(R,(\varpi))$ be the DVR associated to $E$ and $S$ the normalization of $R$ in $L$.
Then it is enough to show that the order of $\varpi$ at every maximal ideal of $S$ is equal to $m$.
Since the order of $\alpha$ along $E$ is equal to $m_0(\frac{m-1}{m}+a)$, there exists an unit $u$ in $R$ such that
\[
\alpha=u \varpi^{m_0(\frac{m-1}{m}+a)}.
\]
Since every coefficient of $\mathrm{div}(\alpha)$ is integer, there exists a positive integer $b$ such that $mb=m_0$.
Let $\rho=m-1+ma$.
Then
\[
\alpha=u \varpi^{b\rho}.
\]
Now, we have
\[
(\frac{T^m}{\varpi^{\rho}})^b = \frac{\alpha}{\varpi^{b\rho}} = u,
\]
so $S$ contains $\frac{T^m}{\varpi^{\rho}}$.
In particular, $R \hookrightarrow S$ factors through $R'=R[Y]/(Y^b-u)$.
Since $K'=K[Y]/(Y^b-u)$ is a field and $R' \to K'$ is injective,
$R'$ is an integral domain and satisfies $R \subset R' \subset S$.
Since $R'$ is \'etale over $R$, $\varpi$ is an uniformizer of $R'_{i}=R'_{\mathfrak{p}_i}$ for every maximal ideal $\mathfrak{p_i}$ of $R'$.
We set $S_i = S \otimes_{R_i} R'_i \subset L$.
Now, we have
\[
(\frac{\varpi^{1+a}}{T})^m=\frac{\varpi^\rho}{T^m} \cdot \varpi = Y^{-1}\varpi.
\]
Since $Y$ is an unit in $R'$, $R_i' \hookrightarrow S_i$ factors through $R''_i=R'_i[Z]/(Z^m-Y^{-1}\varpi)$.
Since $K''=K'[Z]/(Z^m-Y^{-1}\varpi)$ is a field and $R''_i \to K''$ is injective,
$R''_i$ is an integral domain and satisfies $R'_i \subset R''_i \subset S_i$.
Since we have
\[
(\varpi^{1+a}Z^{-1})^{m_0}=(\varpi^{(1+a)m}Z^{-m})^b=(\varpi^{\rho}Y)^b=\alpha,
\]
the quotient field of $R''_i$ is $L$.
Furthermore, since $(Z)$ is the unique maximal ideal of $R''_i$, we obtain $R''_i=S_i$ and $\ord_{S_i}(\varpi)=m$.
Therefore we obtain the last assertion.
\end{proof}

\begin{lem}\label{canonical bundle formula lem}
Let $\pi \colon X \to Y$ be an algebraic fiber space of normal $\Q$-Gorenstein projective varieties.
Assume that $-K_X$ is $\pi$-ample and $X$ is klt.
Then $Y$ is also klt.
\end{lem}

\begin{proof}
There exists an ample Cartier divisor $H$ on $Y$ such that $-K_X+\pi^*H$ is an ample $\Q$-Cartier divisor.
Then for a general element $B \in |-K_X+\pi^*H|_{\Q}$, $(X,B)$ is a klt pair and $K_X+B \sim_{\Q} \pi^*H$.
By Ambro's canonical bundle formula (\cite[Theorem 4.1]{ambro}), we find an effective $\Q$-Weil divisor $B_Y$ such that $(Y,B_Y)$ is a klt $\Q$-Gorenstein pair.
In particular, $Y$ is klt.
\end{proof}

\begin{thm}[\textup{Theorem \ref{covlem}}]\label{covlem'}
Consider the following commutative diagram
\[
\xymatrix{
X \ar[r]^{f} \ar[d]_{\pi} & X \ar[d]^{\pi} \\
Y \ar[r]_{g} & Y, 
}
\]
where $X$ is a klt $\Q$-factorial normal projective varieties, $\pi$ is a Mori fiber space and
$f$ is an int-amplified endomorphism.
Assume that for every irreducible component $E$ of $\Supp R_{f}$, we have $\pi(E)=Y$.  
Then we have the equivariant  commutative  diagram 
\[
\xymatrix{
f \acts X \ar[d]_{\pi}& \ar[l]_{\mu_{X}} \widetilde{X} \ar[d]^{ \widetilde{\pi}} \racts \widetilde{f} \\
g \acts Y & \ar[l]^{\mu_{Y}} A \racts g_{A},
}
\]
where $A$ is an abelian variety, $\mu_{Y}$ is a finite surjective morphism,
$ \widetilde{X}$ is a $\Q$-Gorenstein klt normal projective variety, $\mu_{X}$ is a finite surjective quasi-\'etale morphism, and  
$ \widetilde{\pi}$ is an algebraic fiber space.  
Moreover, $ \widetilde{X}$ is the normalization of the main component of $X {\times}_{Y}A$. 
\end{thm}

\begin{proof}
By Lemma \ref{irreducible fiber lemma}, we can define a positive integer $m_E$ for every prime divisor $E$ on $Y$ so that $\pi^*(E) = m_E F$ for some prime divisor $F$ on $X$.
We set 
\[
\Delta=\sum_E \frac{m_E-1}{m_E}E
\]
as a $\Q$-Weil divisor, where this sum runs over the all prime divisors on $Y$.
Note that this is a finite sum because a general fiber of $\pi$ is reduced. 
Thus $\Delta$ is a well-difined $\Q$-Weil divisor with standard coefficients.
First we prove that $K_Y+\Delta \sim_{\Q} 0$ and $K_Y+\Delta \sim g^*(K_Y+\Delta)$.
Let $E$ be a prime divisor on $Y$ and we set $g^*E=a_1 E_1+ \cdots + a_r E_r$, where $a_1 ,\ldots , a_r$ are positive integers and $E_1, \ldots E_r$ are prime divisors on $Y$.
Then for some prime divisor $F_i$ on X, we have $\pi^*(E_i)=m_{E_i}F_i$ for any $i$.
Since $\pi^*g^*E=f^*\pi^*E$, we have
\[
a_1 m_{E_1} F_1 + \cdots +a_r m_{E_r} F_r = m_E(F_1+ \cdots F_r)
\]
and $a_i m_{E_i}=m_E$ for any $i$.
We remark that $f^*F = F_1 + \cdots F_r$ because any component of $R_f$ is horizontal.
Using this equality, we have
\[
\ord_{E_i}(g^*\Delta-\Delta)=a_i \frac{m_E-1}{m_E}-\frac{m_{E_i}-1}{m_{E_i}} =\frac{m_E-m_{E_i}}{m_{E_i}} = a_i-1 = \ord_{E_i}(R_f).
\]
Hence we obtain $g^*(K_Y+\Delta)\sim K_Y+\Delta$.
Since $g$ is an int-amplified endomorphism, $K_Y +\Delta$ is numerically equivalent to $0$.
Furthermore by the same argument as in the proof of \cite[Lemma 2.11]{brou-hor}, $g$ induces the endomorphism of the log canonical model of $(Y,\Delta)$.
By the same argument as in the proof of \cite[Theorem 1.4]{brou-hor}, $(Y,\Delta)$ is a log canonical pair.
It means that $(Y,\Delta)$ is a log Calabi--Yau pair, and by \cite[Theorem 1.2]{gong}, $K_Y +\Delta$ is $\Q$-linearly equivalent to $0$.

Applying Lemma \ref{index-one covering lem} to the pair $(Y,\Delta)$, 
we obtain the finite covering $\mu_{Y_1} \colon Y_1 \to Y$ as in Lemma \ref{index-one covering lem}.
Since 
\[
\mu_{Y_1}^*g^*(K_Y+\Delta) \sim \mu_{Y_1}^*(K_Y+\Delta) \sim 0,
\]
$g \circ \mu_{Y_1}$ factors $\mu_{Y_1}$.
It means that there exists $g_1 \colon Y_1 \to Y_1$ such that the following diagram commute:
\[
\xymatrix{
Y_1 \ar[r]^{g_1} \ar[d]_{\mu_{Y_1}} & Y_1 \ar[d]^{\mu_{Y_1}} \\
Y \ar[r]_{g} & Y. 
}
\]
Let $X_1$ be the normalization of the main component of $X \times_Y Y_1$.
By Lemma \ref{main comp}, $X_1$ has an int-amplified endomorphism and we get the following equivariant commutative diagram:
\[
\xymatrix{
f \acts X \ar[d]_{\pi}& \ar[l]_{\mu_{X_1}} X_1 \ar[d]^{\pi_1} \racts f_1 \\
g \acts Y & \ar[l]^{\mu_{Y_1}} Y_1 \racts g_{1}.
}
\]
By Lemma \ref{irreducible fiber lemma}, every irreducible component $E$ on $X$ satisfies 
\[\mathrm{codim}(\pi(E)) \leq 1.
\]
Since $\Delta$ has standard coefficients, $K_{Y_1}$ is linearly equivalent to $0$, in particular, $g_1$ is quasi-\'etale.
Hence by Lemma \ref{quasi-etale lemma}, $\mu_{X_1}$ is quasi-\'etale.
It implies that $X_1$ is $\Q$-Gorenstein klt and $-K_{X_1}$ is $\pi_1$-relative ample.
Therefore by Lemma \ref{canonical bundle formula lem}, $Y_1$ is also klt.
By \cite[Theorem 5.2]{meng}, there exists a following commutative diagram:
\[
\xymatrix{
A \ar[r]^{g_A} \ar[d]_{\mu_{Y_2}} & A \ar[d]^{\mu_{Y_2}} \\
Y_1 \ar[r]_{g_1} & Y_1, 
}
\]
where $A$ is an abelian variety, $\mu_{Y_2}$ is a finite surjective morphism and $g_A$ is an int-amplified endomorphism.
Let $\widetilde{X}$ be the normalization of the main component of $X \times_Y A$.
Then $\widetilde{X}$ has an int-amplified endormophism by Lemma \ref{main comp}, and the following diagram commutes:
\[
\xymatrix{
f \acts X \ar[d]_{\pi}& \ar[l]_{\mu_{X}} \widetilde{X} \ar[d]^{\widetilde{\pi}} \racts \widetilde{f} \\
g \acts Y & \ar[l]^{\mu_{Y}} A \racts g_A,
}
\]
where $\mu_Y = \mu_{Y_1} \circ \mu_{Y_2}$.
By Lemma \ref{quasi-etale lemma}, $\mu_X$ is quasi-\'etale.

\end{proof}

\section{Surface case}

In this section, we prove Theorem \ref{main thm} (Theorem \ref{main thm'}).
This section is an only section we assume that $X$ is a surface.
In the smooth case, Theorem \ref{main thm} follows from the classification by Nakayama and Fujimoto in \cite{nak}, \cite{fujimoto}.

First, we consider ruled surfaces over an elliptic curve admitting an int-amplified endomorphism.

\begin{lem}\label{ruled surface over an ellipric curve}
Let $X$ be a minimal ruled surface over an elliptic curve defined over an algebraically closed field of characteristic zero
and $X$ admits an int-amplified endomorphism.
Then $X$ is of dense globally $F$-split type.
\end{lem}

\begin{proof}
By \cite[Theorem 1.2]{amerik} and Proposition \ref{q-etale remain F-splitness}, we may assume that the vector bundle defining $X$ is decomposable.
Then $X$ is of dense globally $F$-split type since every elliptic curve is of dense globally $F$-split type (cf. \cite{ejiri}, \cite{gon-takagi}).
\end{proof}

Next, we consider Mori fiber spaces over a projective line not satisfying the assumption of Theorem \ref{covlem}.

\begin{lem}\label{good fibration}
Consider the following commutative diagram:
\[
\xymatrix{
X \ar[r]^{f} \ar[d]_{\pi} & X \ar[d]^{\pi} \\
\mathbb{P}^1 \ar[r]_{g} & \mathbb{P}^1, 
}
\]
where $X$ is a normal surface, $\pi$ is a Mori fiber space,
$f$ and $g$ are int-amplified endomorphisms.
Assume that there exists a point $P \in \mathbb{P}^1$ suct that $\pi^{-1}(P) \subset \Supp(R_f)$ as sets.
Then $X$ is of Fano type.
\end{lem}

\begin{proof}
If $-K_X$ is not big, then the numerical class of $K_X$ is an eigenvalue of $f^*$ because the Picard rank of $X$ is two.
In particular $\Supp(R_f)$ does not contain any fiber of $\pi$ as sets, so it is contradiction (cf. \cite[Proposition 4.1]{mat-yoshi}). 
Thus $-K_X$ is big and $X$ is a Mori dream space and klt by \cite[Theorem 1.2]{mat-yoshi}.
By \cite{brou-gon}, replacing $f$ by some iterate, we may run the equivariant $(-K_X)$-MMP:
\[
\xymatrix{
X \ar[r]^{f} \ar[d]_{\mu} & X \ar[d]^{\mu} \\
Y \ar[r]_{h} & Y, 
}
\]
where $\mu$ is the $(-K_X)$-negative extremal ray contraction and $Y$ be a normal $\Q$-factorial surface with the Picard rank is equal to one.
Note that $\mu$ contracts a horizontal curve $C$ on $X$ because every vertical curve is $K_X$-negative.
By the assumption, $\Supp(R_h)$ contains the point $\mu(C)$.
By \cite[Theorem 1.2]{brou-hor}, $Y$ is klt.
It implies that $X$ is of Fano type by the argument of \cite[Theorem 1.5]{gost}.
Indeed, since $K_X$ is relative ample, 
$K_X+aC = \mu^*K_Y$ for some positive rational number $a$ by the negativity lemma.
Since $C$ is relative anti-ample, for large enough $n$, $K_X + (a-\frac{1}{n})C$ is anti-ample and $a-\frac{1}{n} > 0$.
Because $Y$ is klt, $(X,(a-\frac{1}{n}C))$ is a klt pair.
Thus, this is a log Fano pair, in particular, $X$ is of Fano type.
\end{proof}

\begin{thm}[Theorem \ref{main thm}]\label{main thm'}
Let $X$ be a normal projective surface defined over an algebraically closed field of characteristic zero
and $X$ admits an int-amplified endomorphism.
Then $X$ is of dense globally $F$-split type.
In particular, $X$ is of Calabi--Yau type.
\end{thm}

\begin{proof}
Let $f$ be an int-emplified endomorphism of $X$.
By \cite[Theorem 1.4]{yoshi} or \cite[Theorem 1.4]{brou-hor},
$X$ is $\Q$-Gorenstein and log canonical.
By \cite[Theorem 1.5]{meng}, $-K_X$ is pseudo-effective.
If $K_X$ is pseudo-effective, $X$ is Q-abelian surface by \cite[Lemma 9.3]{cmz}.
On the other hand, abelian surfaces are of dense globally $F$-split type by {\cite{ogus}}.
By Proposition \ref{q-etale remain F-splitness}, Q-abelian surfaces are of dense globally $F$-split type.

Thus we may assume that $K_X$ is not pseudo-effective.
Replacing $f$ by some iterate, we may run an $f$-equivariant MMP:
\[
X=X_1 \to X_2 \to \cdots \to X_r \to Y,
\]
where $X_i \to X_{i+1}$ is a birational contraction for all $1 \leq i \leq r-1$ and $X_r \to Y$ is a Mori fiber space.
By Theorem \ref{lifting of F-splittness}, it is enough to show that $X_r$ is of dense globally $F$-split type.
In particular, we may assume that $X=X_r$.

We obtain the following diagram:
\[
\xymatrix{
X \ar[r]^{f} \ar[d]_{\pi} & X \ar[d]^{\pi} \\
Y \ar[r]_{g} & Y, 
}
\]
where $\pi$ is a Mori fiber space, $Y$ is an elliptic curve, a projective line or a point, $f$ and $g$ are int-amplified endomorphisms.
If $Y$ is an elliptic curve, then $X$ is a minimal ruled surface over $Y$ by \cite[Proposition 4.1]{mat-yoshi}.
By Lemma \ref{ruled surface over an ellipric curve}, $X$ is of dense globally $F$-split type.
If $Y$ is a point, then $-K_X$ is an ample $\Q$-Cartier divisor and the Picard rank of $X$ is equal to one.
By the proof of \cite[Theorem 5.1]{brou-gon}, $X$ is a projective cone over an elliptic curve or $X$ has at worst rational singularities.
In the first case, $X$ is of dense globally $F$-split type,
and in the second case, Proposition \ref{lc Fano case} implies that $X$ is of dense globally $F$-split type because log canonical surfaces are of dense $F$-pure type by {\cite{MS}} and {\cite{Hara}}.

Hence we may assume that $Y$ is a projective line.
If $R_f$ contains a fiber of $\pi$, then $X$ is of globally $F$-regular type by Lemma \ref{good fibration}.
If $R_f$ does not contain any fiber of $\pi$, above diagram satisfies the assumption of Theorem \ref{covlem} because $X$ is klt $\Q$-factorial by \cite[Lemma 3.7]{mat-yoshi}.
It implies that there exists a following commutative diagram:
\[
\xymatrix{
f \acts X \ar[d]_{\pi}& \ar[l]_{\mu_{X}} \widetilde{X} \ar[d]^{ \widetilde{\pi}} \racts \widetilde{f} \\
g \acts Y & \ar[l]^{\mu_{Y}} A \racts g_{A},
}
\]
where $A$ is an elliptic curve, $\mu_{Y}$ is a finite surjective morphism,
$ \widetilde{X}$ is a $\Q$-Gorenstein klt normal projective variety, $\mu_{X}$ is a finite surjective quasi-\'etale morphism, and  
$ \widetilde{\pi}$ is an algebraic fiber space.  
Moreover, $ \widetilde{X}$ is the normalization of the main component of $X {\times}_{Y}A$. 
By the same argument as above, $\widetilde{X}$ is a minimal ruled surface over $A$.
Hence, $\widetilde{X}$ is of dense globally $F$-split type.
By Proposition \ref{q-etale remain F-splitness}, $X$ is also of dense globally $F$-split type.
Hence in every case, $X$ is of dense globally $F$-split type.
Furthermore, by Theorem \ref{conjecture holds for surface},
$X$ is of Calabi--Yau type.

\end{proof}

Theorem \ref{Fano type thm} follows from the proof of Theorem \ref{main thm} and Theorem \ref{covlem}.

\begin{thm}[\textup{Theorem \ref{Fano type thm}}]\label{Fano type thm'}
Let $X$ be a klt normal projective surface defined over an algebraically closed field of characteristic zero
and admits an int-amplified endomorphism.
Then $X$ is of Fano type if and only if the \'etale fundamental group of the smooth locus $X_{sm}$ of $X$ is finite.
\end{thm}

\begin{proof}
If $X$ is of Fano type, then the \'etale fundamental group of the smooth locus is finite by \cite{gkp}.
On the other hand, we assume that the \'etale fundamental group of the smooth locus is finite.
By Theorem \ref{conjecture holds for surface}, it is enough to show that $X$ is of globally $F$-regular type.
Since $X$ is not Q-abelian, $K_X$ is not pseudo-effective.
Replacing $f$ by some iterate, we may run an $f$-equivariant MMP:
\[
X=X_1 \to X_2 \to \cdots \to X_r \to Y,
\]
where $X_i \to X_{i+1}$ is a birational contraction for all $1 \leq i \leq r-1$ and $X_r \to Y$ is a Mori fiber space.
Since the \'etale fundamental group of every minimal ruled surface over an elliptic curve is infinite, $Y$ is not an elliptic curve.
If $Y$ is a point, then $X$ is of globally $F$-regular type by Corollary \ref{log Fano case}.
Hence we may assume that $Y$ is a projective line.
If $X_r$ is of globally $F$-regular type, then $X$ is of globally $F$-regular type by Lemma \ref{lifting of F-splittness}.
By Lemma \ref{good fibration}, we may assume that $R_{f_r}$ does not contain any fiber of the projective line.

First, we prove $r=1$ in this case.
If $r$ is grater than one, then the exceptional curve of $X_{r-1} \to X_r$ is totally invariant under $f_{r-1}$. 
It means that the image of this exceptional curve is also totally invariant under $f_r$ by Lemma \ref{totinvim}.
In particular, $X_r$ has a totally invariant point.
By using Lemma \ref{totinvim} again, $Y$ has a totally invariant point.
Thus, the fiber of this point is also totally invariant under $f_{r}$,
in particular, $R_{f_r}$ contains this fiber as sets.
this is contradiction to the assumption.

By Theorem \ref{main thm}, there exists a following commutative diagram:
\[
\xymatrix{
f \acts X \ar[d]_{\pi}& \ar[l]_{\mu_{X}} \widetilde{X} \ar[d]^{ \widetilde{\pi}} \racts \widetilde{f} \\
g \acts Y & \ar[l]^{\mu_{Y}} A \racts g_{A},
}
\]
where $A$ is an elliptic curve, $\mu_{Y}$ is a finite surjective morphism and $\mu_{X}$ is a finite surjective quasi-\'etale morphism.
Since $\mu_X$ is quasi-\'etale, $\mu_X$ is \'etale over the smooth locus of $X$.
In particular, the \'etale fundamental group of $\widetilde{X}$ must be finite.
Since $\widetilde{X}$ is a minimal ruled surface over an elliptic curve, it is contradiction.
\end{proof}


\begin{thebibliography}{99}


\bibitem[Ame03]{amerik} E.~Amerik, 
    \textit{On endomorphisms of projective bundle},
     Man. Math. 111 (2003), 17--28.


\bibitem[Amb05]{ambro} F.~Ambro,
    \textit{The moduli b-divisor of an lc-trivial fibration}, Compos. Math. {\textbf{141}} (2005), no. 2, 385–-403.




\bibitem[BG17]{brou-gon} A.~Broustet, Y.~Gongyo, 
	\textit{Remarks on log Calabi--Yau structure of varieties admitting polarized endomorphisms},
	Taiwanese J. Math. {\textbf{21}} (2017), no. 3, 569--582.

\bibitem[BH14]{brou-hor} A.~Broustet,  A.~H\"oring, 
	\textit{Singularities of varieties admitting an endomorphism},
	Math. Ann. {\textbf{360}} (2014), no. 1-2, 439--456.
	
\bibitem[CMZ17]{cmz} P.~Cascini, S.~Meng, D.-Q.~Zhang, 
	\textit{Polarized endomorphisms of normal projective threefolds in arbitrary characteristic},
	arXiv:1710.01903v2.	
	
\bibitem[Eji19]{ejiri} S.~Ejiri,
    \textit{When is the Albanese morphism an algebraic fiber space in positive characteristic?},
    Manuscripta Math. {\textbf{160}} (2019), no. 1-2, 239--264.

	
\bibitem[Fed83]{fed83} R.~Fedder,
    \textit{F-purity and rational singularity}, 
    Trans. Amer. Math. Soc. {\textbf{278}} (1983), no. 2, 461–-480. MR701505 (84h:13031).
    
\bibitem[Fuj17]{fujino} O.~Fujino, 
    \textit{Foundations of minimal model program}, Mathmatical Society of Japan, MSJ Memories, {\textbf{35}} (2017).
    
\bibitem[Fuj02]{fujimoto} Y.~Fujimoto, 
    \textit{Endomorphisms of smooth projective 3-folds with nonnegative Kodaira dimension}, Publ. Res. Inst. Math. Sci. \textbf{38} (2002), 3392

\bibitem[Gon13]{gong} Y.~Gongyo,
	\textit{Abundance theorem for numerically trivial log canonical divisors of semi-log canonical pairs},
	J. Algebraic Geom. {\textbf{22}} (2013), no. 3, 549--564. 
	
\bibitem[GOST15]{gost} Y.~Gongyo, S.~Okawa, A.~Sannai, S.~Takagi, 
    \textit{Characterization of varieties of Fano type via singularities of Cox rings},
    J. Algebraic Geom. {\textbf{24}} (2015), no. 1, 159--182.
    
\bibitem[GS16]{gon-takagi} Y.~Gongyo, S.~Takagi,
    \textit{Surfaces of globally F-regular and F-split type}, 
    Math. Ann. {\textbf{364}} (2016), no. 3--4, 841–-855. 
    
\bibitem[GKP16]{gkp} D.~Greb, S.~Kebekus, T.~Peternell,
    \textit{\'Etale fundamental groups of Kawamata log terminal spaces, flat sheaves, and quotients of Abelian varieties}, Duke Math. J. {\textbf{165}} (2016), no. 10, 1965--2004.
    
\bibitem[Har98]{Hara} N.~Hara, 
Classification of two-dimensional $F$-regular and $F$-pure singularities, 
Adv. Math. \textbf{133} (1998), 33--53. 
    
\bibitem[Har63]{har} R.~Hartshorne,
    \textit{Residues and duality}, Lecture notes of a seminar on the work of A. Grothendieck,
    given at Harvard 1963/64. With an appendix by P. Deligne. Lecture Notes in Mathematics, No. 20,
    Springer-Verlag, Berlin, 
    
\bibitem[HW99]{HH} M.~Hochster, C.~Huneke, 
    Tight closure in equal characteristic zero, Preprint (1999). 
    
\bibitem[HW02]{hw}  N.~Hara, K.-i.~Watanabe,
    $F$-regular and $F$-pure rings vs. log terminal and log canonical singularities, 
    J. Algebraic. Geom. \textbf{11} (2002), no. 2, 363--392. 
	
    
\bibitem[KM98]{komo} J.~Koll\'ar, S.~Mori, 
    \textit{Birational geometry of algebraic varieties}, 
    Cambridge Tracts in Mathematics,134, Cambridge University Press, Cambridge, 1998.
    
\bibitem[Kun86]{kunz} E.~Kunz,
    \textit{K¨ahler differentials}, 
    Advanced Lectures in Mathematics, Friedr. Vieweg $\&$ Sohn, Braunschweig, 1986. MR864975 (88e:14025)

\bibitem[Nak02]{nak} N.~Nakayama, 
    \textit{Ruled surfaces with non-trivial surjective endomorphisms},
    Kyushu J. Math. \textbf{56} (2002), 433– 446.

    
\bibitem[MS19]{mat-yoshi} Y.~Matsuzawa, S.~Yoshikawa, 
    \textit{Int-amplified endomorphism on surfaces},
    arXiv:1902.06071.

\bibitem[Men17]{meng} S.~Meng, 
	\textit{Building blocks of amplified endomorphisms of normal projective varieties},
	arXiv:1712.08995.

\bibitem[MZ1]{meng-zhang} S.~Meng, D.-Q.~Zhang,
	\textit{Building blocks of polarized endomorphisms of normal projective varieties},
	Adv. Math. {\textbf{325}} (2018), 243--273.

\bibitem[MZ2]{meng-zhang2} S.~Meng, D.-Q.~Zhang,
	\textit{Semi-group structure of all endomorphisms of a projective variety admitting a polarized endomorphism},
	arXiv:1806.05828.
	
\bibitem[MS91]{MS} V.~B.~Mehta, V.~Srinivas, 
    Normal $F$-pure surface singularities, 
    J. Algebra {\textbf{143}} (1991), 130--143.
    
\bibitem[MR85]{mr85} V.~B.~Mehta., A.~Ramanathan.,
    \textit{Frobenius splitting and cohomology vanishing for Schubert
    varieties}, Ann. of Math. (2) {\textbf{122}} (1985), no. 1, 27–-40. 
    
    
\bibitem[Ogu81]{ogus} A.~Ogus, 
Hodge cycles and crystalline cohomology, 
Hodge Cycles, Motives, and Shimura Varieties, Lecture Notes in Math., vol. 900, Springer-Verlag, 
Berlin-New York, 1981, 357--414. 



	
\bibitem[PS09]{prokshok-mainII}  Yu.~G.~Prokhorov, V.~V.~Shokurov, 
    Towards the second main theorem on complements, 
    J. Algebraic Geom. {\textbf{18}} (2009), no. 1, 151--199.
	
 \bibitem[SS10]{SS}  K.~Schwede, K.~E.~Smith, 
    Globally $F$-regular and log Fano varieties, 
    Adv. Math. {\textbf{224}} (2010), no. 3, 863--894.
    
\bibitem[ST14]{ST-fin} K.~Schwede, K.~Tucker, 
    \textit{On the behavior of test ideals under finite morphisms}, 
    J. Algebraic Geom. {\textbf{23}} (2014), no. 3, 399–-443. 
    14F18 (14G17)

	

\bibitem[Yos18]{yoshi} S.~Yoshikawa,
	\textit{Singularities of non-$\Q$-Gorenstein varieties admitting a polarized endomorphism},
	arXiv:1811.01795v3.


\end{thebibliography}
\end{document}